\documentclass[12pt,compress]{article}
\usepackage{mathrsfs}
\usepackage{amsfonts}

\usepackage{amsmath}
\usepackage{amssymb}
\usepackage{amsthm}

\usepackage{longtable,booktabs,makecell,multirow,tabularx}
\usepackage{graphicx}

\usepackage{geometry}
\newtheorem{theorem}{Theorem}[section]

\newtheorem{lemma}[theorem]{Lemma}

\newtheorem{corollary}[theorem]{Corollary}

\theoremstyle{definition}

\newtheorem{remark}[theorem]{Remark}

\def\UU{{\mathcal U}}
\def\FF{{\mathcal F}}

\usepackage{latexsym}
\usepackage{amssymb}

\usepackage{color}
\usepackage[colorlinks,citecolor=blue,linkcolor=blue,anchorcolor=blue]{hyperref}
\usepackage[mathscr]{eucal}
\usepackage{cleveref}

\begin{document}


\begin{center}
{\Large \bf     Finite groups with some subgroups satisfying the partial $ \Pi  $-property

\renewcommand{\thefootnote}{\fnsymbol{footnote}}

\footnotetext[1]
{Corresponding author.}

}\end{center}

                        \vskip0.6cm
\begin{center}

                       Zhengtian Qiu, Guiyun Chen and Jianjun Liu$^{\ast}$

                            \vskip0.5cm

       School of Mathematics and Statistics, Southwest University,

       Chongqing 400715, P. R. China

E-mail addresses:  qztqzt506@163.com \, \ gychen1963@163.com    \, \  liujj198123@163.com

\end{center}

                          \vskip0.5cm

\begin{abstract}
	Let $ H $ be a subgroup of a finite group $ G $.  We say that $ H $ satisfies the partial $ \Pi  $-property in $ G $ if  there exists a chief series $ \varGamma_{G}: 1 =G_{0} < G_{1} < \cdot\cdot\cdot < G_{n}= G $ of $ G $ such that for every $ G $-chief factor $ G_{i}/G_{i-1} $ $(1\leq i\leq n) $ of $ \varGamma_{G} $, $ | G / G_{i-1} : N _{G/G_{i-1}} (HG_{i-1}/G_{i-1}\cap G_{i}/G_{i-1})| $ is a $ \pi (HG_{i-1}/G_{i-1}\cap G_{i}/G_{i-1}) $-number.  In this paper, we investigate how some subgroups satisfying the partial $\Pi$-property influence the structure of finite groups.
\end{abstract}

{\hspace{0.88cm} \small \textbf{Keywords:} Finite group, $ p $-nilpotent group, the partial  $ \Pi $-property, saturated formation.}
	
\vskip0.1in

{\hspace{0.88cm} \small \textbf{Mathematics Subject Classification (2020):} 20D10,   20D20.}

\section{Introduction}

    \hspace{0.5cm} All groups considered in this paper are finite.
    We use conventional notions as in \cite{Huppert-1967}. Throughout the paper, $ G $ always denotes a finite group, $ p $ denotes a fixed prime, $ \pi $ denotes some set of primes  and $ \pi(G) $ denotes the set of all primes dividing $ |G| $. An integer $ n $ is called a $ \pi $-number if all prime divisors of $ n $ belong to $ \pi $.

    Recall that a class $ \FF $  of groups is called a formation if $ \FF $ is closed under taking homomorphic images and subdirect products.  A formation $ \FF $ is said to be saturated if $ G/\Phi(G) \in \FF $  implies that $ G \in \FF $.  Throughout  this paper, we  use $ \UU $ (resp. $ \UU_{p} $) to denote the class of supersoluble (resp. $ p $-supersoluble) groups. The supersoluble hypercenter $ Z_{\UU}(G) $  is the the largest normal subgroup of $ G $ such that every $ G $-chief factor below $ Z_{\UU}(G) $ is cyclic,  and the $ p $-supersoluble hypercenter  $ Z_{\UU_{p}}(G) $ is  the largest normal subgroup of $ G $ such that every $ p $-$ G $-chief factor below $ Z_{\UU_{p}}(G) $ is cyclic of order $ p $.


    In \cite{Chen-2013}, Chen and Guo introduced the concept of the partial  $ \Pi $-property of subgroups of finite groups,  which generalizes a large number of known embedding
    properties (see \cite[Section 7]{Chen-2013}).
    A subgroup $H$ of a group $G$ is said to satisfy the partial $ \Pi$-property in $G$  if  there exists a chief series $ \varGamma_{G}: 1 =G_{0} < G_{1} < \cdot\cdot\cdot < G_{n}= G $ of $G$ such that for every $ G $-chief factor $ G_{i}/G_{i-1} $ $ (1\leq i\leq n) $ of $ \varGamma_{G} $, $ | G / G_{i-1} : N _{G/G_{i-1}} (HG_{i-1}/G_{i-1}\cap G_{i}/G_{i-1})| $ is a $ \pi (HG_{i-1}/G_{i-1}\cap G_{i}/G_{i-1}) $-number. They proved the following results by assuming  some maximal  subgroups  satisfiy the partial $ \Pi $-property.

    \begin{theorem}[{\cite[Proposition 1.3]{Chen-2013}}]\label{max}
    	Suppose that $ P $ is a normal $ p $-subgroup of $ G $. If every maximal
    	subgroup of $ P $ satisfies the partial $ \Pi $-property in $ G $, then $ P \leq Z_{\UU}(G) $.
    \end{theorem}

    \begin{theorem}[{\cite[Proposition 1.4]{Chen-2013}}]\label{maximal}
     Let $ E $ be a normal subgroup of $ G $ and let $ P $ be a Sylow $ p $-subgroup of $ E $. If every maximal subgroup of $ P $ satisfies the partial $ \Pi $-property in $ G $, then either $ E\leq Z_{\UU_{p}}(G) $ or $ |E|_{p}=p $.
    \end{theorem}

    In this paper, we mainly consider the question: if a finite group $ G $ does have some subgroups satisfying the partial $ \Pi $-property, what can we say about $ G $? Our first result is as follows.

    \begin{theorem}\label{first}
    	Let $ p $ be a prime dividing the order of $ G $ with $ (|G|, p-1) = 1 $ and $ P \in \mathrm{Syl}_{p}(G) $.
    	Then $ G $ is $ p $-nilpotent if and only if every maximal subgroup of $ P $ satisfies the partial $ \Pi $-property in $ N_{G}(P) $ and $ P' $ satisfies the partial $ \Pi $-property in $ G $.
    \end{theorem}

    \begin{remark}
    	 The hypothesis that ``$ P' $ satisfies the partial $ \Pi $-property in $ G $'' in Theorem \ref{first} is essential. Let $ G = PSL(2, 7) $ and $ P $ a Sylow $ 2 $-subgroup of $ G $.  It is not difficult to  see that  every maximal subgroup of $ P $ satisfies the partial $ \Pi $-property in
    	$ N_{G}(P)=P $ and $ P' $ does not satisfy the partial $ \Pi $-property in $ G $. However, $ G $ is not $ 2 $-nilpotent or even $ 2 $-soluble.
    	
    	In fact, we cannot remove the hypothesis that ``$ p $ is a prime dividing the order of $ G $ with
    	$ (|G|, p-1) = 1 $'' in Theorem \ref{first}. Let $ P $ be a Sylow $ 3 $-subgroup of $ G = A_{5} $, the alternating
    	group of degree $ 5 $. Then every maximal subgroup of $ P $ satisfies the partial $ \Pi $-property in
    	$ N_{G}(P) $ and $ P' $ satisfies the partial $ \Pi $-property in $ G $. However, $ A_{5} $ is not $ 3 $-nilpotent.
    \end{remark}

    Based on Theorem \ref{first}, we can prove the following theorem.



   \begin{theorem}\label{second}
   	Let $ p $ be a prime dividing the order of $ G $ with $ (|G|, p-1) = 1 $ and let $ N $ a normal subgroup of $ G $ such that $ G/N $ is $ p $-nilpotent. Let $ P \in \mathrm{Syl}_{p}(N) $ and suppose that every
   	maximal subgroup of $ P $ satisfies the partial $ \Pi $-property in $ N_{G}(P) $ and $ P' $ satisfies the partial $ \Pi $-property in $ G $. Then $ G $ is $ p $-nilpotent.
   \end{theorem}

At last, we extend the above theorem to formation case.

   \begin{theorem}\label{third}
   	Let $ \FF $ be a saturated formation containing the class of all supersoluble groups
   	$ \UU $ and let $ N $ be a normal subgroup of a group $ G $ such that $ G/N\in \FF $. Suppose that, for all
   	primes $ p $ dividing the order of $ N $ and for all $ P \in \mathrm{Syl}_{p}(N) $, every maximal subgroup of $ P $
   	satisfies the partial $ \Pi $-property in $ N_{G}(P) $ and suppose that $ P' $ satisfies the partial $ \Pi $-property in $ G $. Then $ G \in \FF $.
   \end{theorem}

\section{Preliminaries}


In this section, we will restate some basic lemmas, which are required in the proofs of our main results.

  \begin{lemma}[{\cite[Lemma 2.1(3)]{Chen-2013}}]\label{over}
  	Let $ H \leq G $ and $ N \unlhd G $. If either $ N \leq H $ or $ (|H|, |N|)=1 $ and $ H $ satisfies the partial $ \Pi $-property in $ G $, then $ HN/N $ satisfies the partial $ \Pi $-property in $ G/N $.
  \end{lemma}

  \begin{lemma}\label{subgroup}
  	Let  $ H\leq N \leq G $.  If $ H $ is a $ p $-subgroup of $ G $ and $ H $ satisfies the partial $ \Pi $-property in $ G $, then $ H $ satisfies the partial $ \Pi $-property in $ N $.
  \end{lemma}

   \begin{proof}
  	By hypothesis, there exists a chief series $ \varGamma_{G}: 1 =G_{0} < G_{1} < \cdot\cdot\cdot < G_{n}= G $ of $ G $ such that for every $ G $-chief factor $ G_{i}/G_{i-1}$ $ (1\leq i\leq n) $ of $ \varGamma_{G} $, $ | G  : N _{G} (HG_{i-1}\cap G_{i})| $ is a $ p $-number. Therefore, $ \varGamma_{N} : 1=G_{0}\cap N\leq G_{1}\cap N\leq\cdot\cdot\cdot <G_{n}\cap N=N $ is, avoiding repetitions, a normal series of $ N $.  For any  normal section $ (G_{i}\cap N)/(G_{i-1}\cap N) $ $ (1\leq i\leq n) $ of $ N $,   we have that $ H(G_{i-1}\cap N)\cap (G_{i}\cap N)=(H\cap G_{i})G_{i-1}\cap N $. Note that $ | N  : N\cap N _{G} (HG_{i-1}\cap G_{i})| $ is a $ p $-number, and so $ | N  : N _{N} (HG_{i-1}\cap G_{i}\cap N)| $ is a $ p $-number.  Since $ H(G_{i-1} \cap N)\cap (G_{i}\cap N)=(H\cap G_{i})G_{i-1}\cap N $, it follows that $ | N  : N _{N} (H(G_{i-1} \cap N)\cap (G_{i}\cap N))| $ is a $ p $-number. Let  $ A/B$ be a chief factor of $ N $ such that $ G_{i-1}\cap N\leq B\leq A\leq G_{i}\cap N $. Then $ N _{N} (H(G_{i-1} \cap N)\cap (G_{i}\cap N))\leq N _{N} ((H(G_{i-1} \cap N)\cap A)B)=N_{N}(HB\cap A) $. Hence $ |N:N_{N}(HB\cap A)| $ is a $ p $-number. This  means that  $ H $ satisfies the partial $ \Pi $-property in $ N $.
  \end{proof}

\begin{lemma}\label{NG}
	 Let $ N $ be a normal subgroup of a group $ G $ and let $ P $ be a $ p $-subgroup
	of $ G $. Then $ N_{G/N}(PN/N) = N_{G}(P)N/N $ if one of the following holds:
	
	{\rm (1)} $ P $ is a Sylow $ p $-subgroup of $ G $.
	
    {\rm (2)} $ (|N|, p)=1 $.
\end{lemma}

\begin{proof}
	 The result immediately follows from \cite[Kapitel I, Hilfssatz 7.7]{Huppert-1967} and \cite[Lemma 7.7]{isaacs2008finite}
\end{proof}


    \begin{lemma}\label{also}
    	 Let $ K $ be a normal subgroup of   $ G $ and let $ P $ be a
    	Sylow $ p $-subgroup of $ G $.  If every maximal subgroup of $ P $ satisfies the partial $ \Pi $-property in $ N_{G}(P) $, then  every maximal subgroup of $ PK/K $ satisfies the partial $ \Pi $-property in $ N_{G/K}(PK/K) $.
    \end{lemma}

    \begin{proof}
    	Write $ U=N_{G}(P) $.  Let $ L/K $ be a maximal subgroup of $ PK/K $. Then $ L = K(L \cap P) $ and $ P \cap L $ is a maximal subgroup of $ P $. Set $ P_{1} = P \cap L $.  By hypothesis, there
    	exists a chief series $$  1 =U_{0} < U_{1} < \cdot\cdot\cdot < U_{n}=U=N_{G}(P) $$ of $U$ such that for every chief factor $ U_{i}/U_{i-1} $ $ (1\leq i\leq n) $, $ | U  : N _{U} (P_{1}U_{i-1}\cap U_{i})| $ is a $ p $-number. It is easy
    	to see that $$  1 \leq U_{1}K/K \leq \cdot\cdot\cdot \leq U_{n}K/K=UK/K=N_{G}(P)K/K $$ is a normal series of $ UK/K $. Write $ \overline{G}=G/K $. By Lemma \ref{NG}, $ \overline{N_{G}(P)}=N_{\overline{G}}(\overline{P}) $. It suffices to show that  $ |\overline{N_{G}(P)}:N_{\overline{N_{G}(P)}}(\overline{P_{1}U_{i-1}}\cap \overline{U_{i}})| $ is a $ p $-number. Since $ U $ is $ p $-soluble, we have that  $ U_{i} /U_{i-1} $ is a $ p $-group or $ p' $-group.   If $ U_{i} /U_{i-1} $ is a $ p' $-group, then we are done. Now suppose that $ U_{i}/U_{i-1} $ is a $ p $-group.  Notice that $ |U:N_{U}(P_{1}U_{i-1}\cap U_{i})| $ is a $ p $-number. It follows that $ P_{1}U_{i-1}\cap U_{i}\unlhd U $, and thus either $ P_{1}U_{i-1}\cap U_{i}=U_{i-1} $ or $ P_{1}U_{i-1}\cap U_{i}=U_{i} $. If  $ P_{1}U_{i-1}\cap U_{i}=U_{i} $, then $ \overline{P_{1}U_{i-1}} \cap \overline{U_{i}} =\overline{U_{i}} $, and so $ |\overline{N_{G}(P)}:N_{\overline{N_{G}(P)}}(\overline{P_{1}U_{i-1}}\cap \overline{U_{i}})|=1 $ is a $ p $-number, as desired. If  $ P_{1}U_{i-1}\cap U_{i}=U_{i-1} $, then $ P_{1}U_{i-1}/U_{i-1}\cap U_{i}/U_{i-1}=1 $. Since $ U_{i}/U_{i-1} $ is a $ p $-group, we have that  $ P_{1}U_{i-1}/U_{i-1} $ is a maximal subgroup of $ PU_{i-1}/U_{i-1} $, and thus $ |U_{i}/U_{i-1}|=p $. Hence $ \overline{U_{i}}/\overline{U_{i-1}}\cong U_{i}K/U_{i-1}K $ has order at most $ p $. So $ \overline{P_{1}U_{i-1}}/\overline{U_{i-1}}\cap \overline{U_{i}}/\overline{U_{i-1}}=\overline{U_{i}}/\overline{U_{i-1}} $ or $ 1 $.   In any case, we have that $ |\overline{N_{G}(P)}:N_{\overline{N_{G}(P)}}(\overline{P_{1}U_{i-1}}\cap \overline{U_{i}})|=1 $, as desired. This completes the proof of the lemma.
    \end{proof}

\begin{lemma}\label{satisfies}
	 Let $ H $ be a normal subgroup of $ G $ and let $ P $ be a
	Sylow $ p $-subgroup of $ H $. If $ K $ is a normal $ p' $-subgroup of $ G $ and every maximal
	subgroup of $ P $ satisfies the partial $ \Pi $-property in $ N_{G}(P) $, then  every maximal subgroup of $ PK/K $ satisfies the partial $ \Pi $-property in $ N_{G/K}(PK/K) $.
\end{lemma}

\begin{proof}
	 Write $ \overline{G}=G/K $ and $ U=N_{G}(P) $. Clearly, $ N_{\overline{G}}(\overline{P}) =\overline{N_{G}(P)} $.   Let $ P_{1} $ be a maximal subgroup of $ P $ and $ U_{i} /U_{i-1} $ $ (1 \leq i \leq n) $ a chief factor of $ N_{G}(P) $ as in the proof of Lemma \ref{also}. It is easy
	 to see that $$  1 \leq \overline{U_{1}} \leq \cdot\cdot\cdot \leq \overline{U_{n}}=\overline{U}=\overline{N_{G}(P)} $$ is a normal series of $ UK/K $.
	
	 Now we consider the index $ |P_{1}U_{i}K : P_{1}U_{i-1}K| $. On the one hand, we have $ |P_{1}U_{i}K : P_{1}U_{i-1}K| = |U_{i}/U_{i-1}||P_{1} \cap U_{i-1}K : P_{1} \cap U_{i}K| $. On the other hand, $ |P_{1}U_{i}K : P_{1}U_{i-1}K| =
	 |U_{i} /U_{i-1}||K\cap P_{1}U_{i-1} : K\cap P_{1} U_{i} | $.  Therefore $ |P_{1} \cap U_{i}K : P_{1} \cap U_{i-1}K| = |K\cap P_{1} U_{i} : K\cap P_{1} U_{i-1} | $.
	 Observe that $ |P_{1}\cap  U_{i}K : P_{1} \cap U_{i-1}K | $ is a $ p $-number and $ |K \cap P_{1} U_{i} : K \cap P_{1}U_{i-1} | $ is a
	 $ p' $-number. It follows that $ |P_{1}\cap  U_{i}K : P_{1} \cap U_{i-1}K| = 1 $, and thus $ P_{1} \cap U_{i}K =P_{1}\cap U_{i-1} K $. Furthermore,  $ P_{1}U_{i-1}K \cap U_{i}K =P_{1}U_{i-1}K\cap U_{i-1} K=U_{i-1}K $. Hence $ |\overline{N_{G}(P)}: N_{\overline{N_{G}(P)}} (\overline{P_{1}U_{i-1}}\cap \overline{U_{i}})| = 1 $. This means
	 that $ \overline{P_{1}} $ satisfies the partial $ \Pi $-property in $ N_{\overline{G}}(\overline{P}) $.
\end{proof}

\begin{lemma}[{\cite[Lemma 2.6]{GuoX-2000}}]\label{product}
	 Let $ N $ be a soluble normal subgroup of $ G $ with $ N \not = 1 $. If every minimal normal subgroup of $ G $ which is contained in $ N $ is not contained in $ \Phi(G) $, then the Fitting subgroup $ F(N) $ of $ N $ is the direct product of minimal normal subgroups of $ G $ which are contained in $ N $.
\end{lemma}

\begin{lemma}[{\cite[Kapitel IV, Satz 4.7]{Huppert-1967}}]\label{normal}
	Let $ P $ be a Sylow $ p $-subgroup of $ G $ and $ N \unlhd G $. If
	$ P \cap N \leq \Phi(P) $, then $ N $ is $ p $-nilpotent.
\end{lemma}

\begin{lemma}[{\cite[Lemma 2.16]{skiba2007weakly}}]\label{in}
	 Let $ \FF $ be a saturated formation containing the class of all
	supersoluble groups. Let $ E $ be a normal subgroup of $ G $ such that $ G/E\in \FF $. If $ E $ is cyclic, then $ G\in \FF $.
\end{lemma}

\begin{lemma}\label{p-nilpotent}
	 Let $ H $ be a $ p $-subgroup of $ G $. If $ G $ is $ p $-nilpotent, then $ H $ satisfies the partial $ \Pi $-property in  $ G $.
\end{lemma}

\begin{proof}
	Let $ A/B $ be an arbitrary chief factor of $ G $. Then $ |A/B| $ is a $ p' $-number or $ |A/B| = p $.
	If $ |A/B| $ is a $ p' $-number, then $ (H \cap A)B/B = 1 $. If $ |A/B| = p $, then $ (H \cap A)A/B = A/B $ or $ 1 $.
	In any case, we have that $ |G/B : N_{G/B}((H \cap A)B/B)| = 1 $. Therefore $ H $ satisfies the partial $ \Pi $-property in $ G $, as wanted.
\end{proof}

  	
  	
  	




  \section{Proofs}



  \begin{proof}[Proof of Theorem \ref{first}]
  	By Lemma \ref{p-nilpotent}, we only need to prove the sufficiency. Assume that
  	$ G $ is not $p$-nilpotent and $ G $ is a counterexample of minimal order. We divide the proof into the following steps.
  	
  	\vskip0.1in
  	
  	\noindent\textbf{Step 1.} $ O_{p'}(G) = 1 $.
  	
  	\vskip0.1in
  	
  	Assume that $ O_{p'}(G) > 1 $. Set $ \overline{G} = G/O_{p'}(G) $. By Lemmas \ref{over} and \ref{NG}, $ (\overline{P})' = \overline{P'} $ satisfies
  	the partial $ \Pi $-property in $ N_{\overline{G}}(\overline{P}) = \overline{N_{G}(P)} $. Furthermore, every maximal subgroup of $ P $
  	satisfies the partial $ \Pi $-property in $ G $. Hence $ \overline{G} $ satisfies the hypotheses of the theorem. The
  	minimal choice  of $ G $ yields that $ \overline{G} $ is $ p $-nilpotent, and thus $ G $ is $ p $-nilpotent, a contradiction.
  	
  	\vskip0.1in
  	
  	\noindent\textbf{Step 2.}  If $ H $ is a proper subgroup of $ G $ with $ P \leq H $, then $ H $ is $ p $-nilpotent. In particular,
  	$ N_{G}(P) $ is $ p $-nilpotent.
  	
  	\vskip0.1in
  	
  	Clearly, $ N_{H}(P) \leq N_{G}(P) $. By Lemma \ref{subgroup}, every maximal subgroup of $ P $ satisfies the partial $ \Pi $-property in $ N_{H}(P) $ and $ P' $ satisfies the partial $ \Pi $-property in $ H $. This means that
  	$ H $ satisfies the hypotheses of the theorem. The minimal choice of $ G $ implies that $ H $ is $ p $-nilpotent. If $ N_{G}(P) = G $, then $ G $ is $ p $-supersoluble according to Theorem \ref{max}. By \cite[Chapter 2, Lemma 5.25]{Guo2015}, $ G $ is $ p $-nilpotent, a contradiction. Therefore, $ N_{G}(P) < G $, and so $ N_{G}(P) $ is $ p $-nilpotent.
  	
  	\vskip0.1in
  	
  	\noindent\textbf{Step 3.} $ O_{p}(G) > 1 $ and there exists a minimal normal subgroup $ K $ of $ G $ such that $ G/K $ is
  	$ p $-nilpotent. Moreover, $ K \leq O_{p}(G) $ and $ G $ is $ p $-soluble..
  	
  	\vskip0.1in
  	
  	If $ P'=1 $, then $ P $ is abelian. Since $ N_{G}(P) $ is $ p $-nilpotent, we have that $ P \leq Z(N_{G}(P)) $. By
  	Burnside's Theorem (see \cite[Theorem 5.13]{isaacs2008finite}), it follows that $ G $ is $ p $-nilpotent, a contradiction. Thus $ P' > 1 $. By the hypothesis, there exists a chief series
  	\begin{equation}\label{1}
  		\varGamma_{G}: 1 =G_{0} < G_{1} < \cdot\cdot\cdot < G_{n}= G
  	\end{equation}
  	of $ G $ such that for every $ G $-chief factor $ G_{i}/G_{i-1}$ $ (1\leq i\leq n) $ of $ \varGamma_{G} $, $ | G  : N _{G} (P'G_{i-1}\cap G_{i})| $ is a $ p $-number. Set $ G_{1} = K $. Hence $ |G : N_{G}(P' \cap K)| $ is a $ p $-number, and so $ P' \cap K \unlhd G $. We
  	claim that $ K \leq O_{p}(G) $. If $ P' \cap K = K $, then $ K \leq O_{p}(G) $, as wanted. So we can assume that $ P' \cap K = 1 $. Furthermore, $ P \cap K $ is abelian. It turns out that $ P \cap K $ is abelian. If
  	$ N_{G}(P \cap K) = G $, then $ P \cap K $ is normal in $ G $. By Step 1 and the minimal normality of $ K $, we have that $ K \leq P $, and therefore $ K\leq O_{p}(G) $, as wanted. Hence we may assume that
  	$ N_{G}(P \cap K) $ is a proper subgroup of $ G $. Since $ P \leq N_{G}(P \cap K) $, we have that $ N_{G}(P \cap K) $ is $ p $-nilpotent by Step 2, and so $ N_{K}(P\cap K) $ is $ p $-nilpotent. It follows that $ P\cap K \leq Z(N_{K}(P \cap K)) $. Applying Burnside's Theorem (see \cite[Theorem 5.13]{isaacs2008finite}), we know that $ K $ is $ p $-nilpotent. In
  	view of Step 1, $ K\leq O_{p}(G) $, as claimed.
  	
  	According to (\ref{1}), we can see that the following series $$ 1 < G_{2}/K = G_{2}/G_{1} < \cdots < G_{n}/G_{1} = G/G_{1} $$
  	is a chief series of $ G/G_{1} $. For every $ G $-chief factor $ G_{i}/G_{i-1} $  $ (2 \leq i \leq n) $, $ |G : N_{G}(P'G_{i-1}\cap G_{i})| $
  	is a $ p $-number. Thus $ P'G_{i-1} \cap G_{i} \unlhd G $.  Observe that $ G_{i-1}\leq P'G_{i-1} \cap G_{i} \leq G_{i} $. If $ G_{i-1} =
  	P'G_{i-1} \cap G_{i} $, then $ G_{i-1}K/K = P'G_{i-1}K/K \cap G_{i}K/K $. If $  G_{i}= P'G_{i-1} \cap G_{i} $, then $ G_{i}K/K =
  	P'G_{i-1}K/K \cap G_{i}K/K $. In any case, $ |G/K : N_{G/K}(P'G_{i-1}K/K \cap G_{i}K/K)| = 1 $. This means that $ P'K/K $ satisfies the partial $ \Pi $-property in $ G/K $. Applying Lemma \ref{over}, every maximal subgroup of $ P/K $ satisfies the partial $ \Pi $-property in $ N_{G}(P)/K = N_{G/K}(P/K) $. Hence $ G/K $ satisfies the hypotheses of the theorem. The minimal choice  of $ G $ implies that $ G/K $ is $ p $-nilpotent. Therefore $ G $ is $ p $-soluble, as desired.

  	

  	\vskip0.1in
  	
  	\noindent\textbf{Step 4.}  There exists a Sylow $ q $-subgroup $ Q $ of $ G $ with $ q \not = p $ such that $ G = PQ $.
  	
  	\vskip0.1in
  	
  	Since $ G $ is $ p $-soluble, there exits a Sylow $ q $-subgroup $ Q $ of $ G $ such that $ PQ = QP $ for any
  	prime $ q \not = p $ by \cite[Chapter 6, Theorem 3.5]{Gorenstein}. If $ PQ < G $, then $ PQ $ is $ p $-nilpotent by Step 2.
  	In view of \cite[Theorem 3.21]{isaacs2008finite}, we have $ Q \leq C_{G}(O_{p}(G)) \leq O_{p}(G) $, a contradiction
  	
  	\vskip0.1in
  	
  	\noindent\textbf{Step 5.} $ N_{G}(P) = P $ is a maximal subgroup of $ G $.
  	
  	\vskip0.1in
  	
  	Let $ L $ be a maximal subgroup of $ G $ such that $ N_{G}(P) \leq L $. By Step 2, $ L $ is $ p $-nilpotent
  	and therefore $ O_{p'}(L) \leq C_{G}(O_{p}(G)) \leq O_{p}(G) $. Hence  $ N_{G}(P) = P=L $ is a maximal subgroup of $ G $.

  	
  	\vskip0.1in
  	
  	\noindent\textbf{Step 6.} $ QK/K $ is a minimal normal subgroup of $ G/K $ and $ Q $ is an elementary abelian group.
  	
  	\vskip0.1in
  	
  	Notice that $ G/K $ is $ p $-nilpotent and $G=PQ$. Let $ T/K $ be a minimal normal subgroup of $ G/K $ contained in $ QK/K $. By  Step 5, we conclude that $G=PT$ and $ T=QK $. Thus Step 6 follows.
  	
  	\vskip0.1in
  	
  	\noindent\textbf{Step 7.} $ P' \cap K = 1 $ and $ K \leq Z(P) $.
  	
  	\vskip0.1in
  	
  	For the $ G $-chief factor $ K/1 $, $ |G : N_{G}(P' \cap K)| $ is a $ p $-number, and so $ P' \cap K \unlhd G $. The
  	minimal normality  of $ K $ implies that $ P' \cap K = K $ or $ 1 $. If $ P'\cap K = K $, then $ P\cap KQ = K \leq P' \leq \Phi(P) $,
  	and thus $ KQ $ is $ p $-nilpotent by Lemma \ref{normal}. Thus $ G $ is $ p $-nilpotent, a contradiction.  Therefore $ K \cap P' = 1 $. Since $ [P, K] \leq [P,P] = P' $, we have that $ [P, K] \leq P' \cap K = 1 $, as desired.
  	
  	\vskip0.1in
  	
  	\noindent\textbf{Step 8.} $ Q $ has no fixed points on $ K $ and $ |Q| = q $.
  	
  	\vskip0.1in
  	
  	If there exists a non-identity element $ a $ in $ K $ and a non-identity element $ b $ in $ Q $ such that
  	$ [a, b] = 1 $, then $ \langle P, b\rangle  \leq C_{G}(a) $. By Step 5, we know that $ G = \langle P, b\rangle = C_{G}(a) $. The
  	minimal normality of $ K $ yields that $ K = \langle a \rangle $. In this case, $ KQ = K \times Q $, and hence $ G $ is $p$-nilpotent, a contradiction. Thus $ Q $ has no fixed points on $ K $. By \cite[Theorem 8.3.2]{Kurzweil}, $ Q $ is a cyclic group of order $ q $.
  	
  	\vskip0.1in
  	
  	\noindent\textbf{Step 9.} The final contradiction.
  	
  	\vskip0.1in
  	
  	Since $ KQ $ is normal in $ G $, we have $ G = KN_{G}(Q) $ by the Frattini argument. As $ K $ is
  	elementary abelian, we see that $ K\cap N_{G}(Q) $ is normal in $ G $. It follows that $ K\cap N_{G}(Q) = 1 $ or
  	$ K $. If $ K \cap N_{G}(Q) = K $, then $ Q $ is a normal $ p $-complement in $ G $, a contradiction. Therefore,
  	$ K \cap N_{G}(Q) = 1 $. Let $ S \in \mathrm{Syl}_{p}(N_{G}(Q)) $ such that $ S \leq P $. By Step 7, we have $ S \leq C_{G}(K) $.
  	Since $ S $ normalizes $ Q $ and $ C_{G}(K) \unlhd G $, it follows from Step 8 that $ [S, Q] \leq C_{G}(K) \cap  Q = 1 $.
  	Thus $ S \leq  C_{G}(Q) $. By Step 4, we deduce that $ N_{G}(Q) = SQ = C_{G}(Q) $. By Burnside's
  	Theorem (see \cite[Theorem 5.13]{isaacs2008finite}), we see that $ G $ is $ q $-nilpotent. It forces that $ G = N_{G}(P) $,
  	which contradicts Step 5. This final contradiction completes the proof.
  \end{proof}

    \begin{proof}[Proof of Theorem \ref{second}]
    	We argue by induction on $ |G| $. By Lemma \ref{subgroup}, every maximal subgroup of $ P $ satisfies the partial $ \Pi $-property in $ N_{N}(P) $, and $ P' $ satisfies the partial $ \Pi $-property in $ N $. By Theorem \ref{first}, $ N $ is $ p $-nilpotent.  Let $ L $ be a normal $ p $-complement in $ N $. Then $ L \unlhd G $ and $ (G/L) \big / (N/L) $ is $ p $-nilpotent. Applying Lemmas \ref{over} and \ref{satisfies}, we know  that $ G/L $ satisfies the hypotheses of the theorem.
    	
    	If $ L > 1 $, then $ G/L $ is $ p $-nilpotent by induction and so $ G $ is $ p $-nilpotent. Hence we may
    	assume that $ L = 1 $ and therefore $ N = P $. Let $ T/N $ be a normal $ p $-complement
    	in $ G/N $. By the  Schur-Zassenhaus theorem, there exists  a complement $ T_{1} $ of $ N $ in $ T $, and $ T = T_{1}N $.
    	It is clear that $ T $ satisfies the hypotheses of Theorem \ref{first}. Therefore $ T = T_{1}\times  N $. Hence $ T_{1} $ is a normal $ p $-complement in $ G $. This completes the proof.
    \end{proof}

    Recall that a group $ G $ is a Sylow tower group of supersoluble type if, for $ p_{1} > p_{2} > \cdots > p_{s} $ are the distinct prime divisors of the order of $ G $, there exists a series of normal subgroups of $ G $, $$ 1 = G_{0} \leq G_{1} \leq \cdots \leq G_{s} = G $$ such that $ G_{i} /G_{i-1} $ is a Sylow $ p_{i} $-subgroup of $ G_{i}/G_{i-1} $ for $ i = 1, ... , s $.


\begin{corollary}\label{co}
	For every Sylow subgroup $ P $ of a group $ G $, suppose that every maximal
	subgroup of $ P $ satisfies the partial $ \Pi $-property in $ N_{G}(P) $ and $ P' $ satisfies the partial $ \Pi $-property in $ G $. Then $ G $ is a Sylow tower group of supersoluble type.
\end{corollary}

\begin{proof}
	Let $ p $ be the smallest prime dividing $ |G| $ and $ P $ a Sylow $ p $-subgroup of $ G $. By Theorem \ref{first}, $ G $ is $ p $-nilpotent. Let $ L $ be a normal $ p $-complement of $ G $. It is clear that $ L $ satisfies  the
	hypotheses of $ G $,  and so $ L $ is a Sylow tower group of supersoluble type by induction.  This shows that $ G $ is a Sylow tower group of supersoluble type.
\end{proof}

  \begin{proof}[Proof of Theorem \ref{third}]
  	Assume that the theorem is false and let $ G $ be a counterexample of
  	minimal order.
  	
  	In view of Lemma \ref{subgroup} and Corollary \ref{co}, $ N $ is a Sylow tower group of supersoluble type.
  	Let $ p $ be the largest prime dividing the order of $ N $ and let $ P $ be a Sylow $ p $-subgroup of $ N $. Then $ P\unlhd G $ and $ G/P \big / N/P\cong G/N\in \FF $.  By Lemmas \ref{over}, \ref{NG} and \ref{satisfies}, we see that $ G/P $ satisfies the hypotheses of the theorem. The minimal choice  of $ G $ yields that $ G/P \in \FF $.

    Assume that $ \Phi(P) > 1 $. Clearly, $ \Phi(P) \unlhd G $ and $ G/\Phi(P) \big / P/\Phi(P) \cong G/P $. By Lemmas \ref{over}, \ref{NG} and \ref{satisfies}, we see that $ G/\Phi(P) $
    satisfies the hypotheses of our theorem . Thus $ G/\Phi(P) \in \FF $ by the minimal choice of $ G $. Then $ G/\Phi(G) \cong G/\Phi(P) \big / \Phi(G)/\Phi(P) \in \FF  $ and so $ G\in \FF $, a contradiction. Hence $ \Phi(P) = 1 $.
  	
  	Let $ T $ be a minimal normal subgroup of $ G $ contained in $ P $. By Lemmas \ref{over}, \ref{NG} and \ref{satisfies},  $ G/T $ satisfies the	hypotheses of the theorem.  The minimal choice  of $ G $ yields that $ G/T \in \FF $. Note that $ T \nleq  \Phi(G) $. By Lemma \ref{product}, we may assume that
  	$$ P = T_{1} \times T_{2} \times \cdots \times T_{s} ,  $$ where $ T_{1} ,..., T_{s}  $ are some minimal normal subgroups of $ G $. By the above argument, $ G/T_{j}\in \FF $
  	for all $ j \in \{ 1, ... , s \} $. If $ s > 1 $, then $ G\cong G/(T_{1}\cap T_{2})\in \FF $, a contradiction. Hence $ s=1 $ and $ P = T_{1} $ is a minimal normal subgroup of $ G $.
  	
  	Let $ S $ be a Sylow $ p $-subgroup of $ G $. Choose a maximal subgroup $ P_{1} $ of $ P $ such that $ P_{1} \unlhd S $.
  	By the hypothesis, there exists a chief series
  	
  	$$ \varGamma_{G}: 1 =G_{0} < G_{1} < \cdot\cdot\cdot < G_{n}= G $$
  	of $ G = N_{G}(P) $ such that $ |G : N_{G}(P_{1}G_{i-1} \cap G_{i})| $ is a $ p $-number for every $ G $-chief factor $ G_{i}/G_{i-1} $ $(1\leq i\leq n) $. There exists an integer $ k $ such that $ P\leq  G_{k}P_{1}  $, but $ P\nleq G_{k-1}P_{1} $,
  	where $ 1 \leq k \leq n $. Notice that $ |G : N_{G}(P_{1}G_{k-1} \cap G_{k})| $ is a $ p $-number. Since $ P_{1}\unlhd S $, we
  	have $ P_{1} G_{k-1} \cap G_{k} \unlhd  G $. If $ P_{1}G_{k-1} \cap  G_{k}= G_{k} $, then $ P \leq  G_{k}P_{1}=G_{k-1}P_{1}  $, a contradiction.
  	Therefore, $ P_{1}G_{k-1} \cap G_{k}= G_{k-1} $, and so $ P_{1}\cap G_{k}= P_{1} \cap G_{k-1} $. Observe that $ P = P \cap G_{k}P_{1} =
  	(P \cap  G_{k})P_{1} $. The minimal normality of $ P $ implies that $ P \cap G_{k}= P $. Furthermore, $ P_{1} =
  	P_{1} \cap G_{k}= P_{1} \cap  G_{k-1} $. Since $ P \nleq G_{k-1}P_{1} $, we have that $ P \nleq  G_{k-1} $. Thus $ P \cap G_{k-1}= P_{1} $.
  	The minimal normality of $ P $ yields that $ P_{1}= 1 $ and $ |P| = p $. In view of Lemma \ref{in}, $ G\in \FF $, a contradiction. Our proof is now complete.
  	\end{proof}

  	

  	
  	
  	
  	
  	
  	


\section*{Acknowledgments}

    \hspace{0.5cm} This work is supported by the National Natural Science Foundation of China (Grant No.12071376, 11971391) and  the  Fundamental Research Funds for the Central Universities (No. XDJK2020B052).

\end{document}